\documentclass[11pt]{amsart}

\usepackage{amscd,amsmath,amssymb,fancyhdr,color}
\usepackage[hidelinks]{hyperref}
\usepackage[utf8x]{inputenc}
\usepackage{amsfonts}
\usepackage{amsthm}
\usepackage{accents}
\usepackage{graphicx}
\usepackage{float}
\usepackage{amssymb}
\usepackage{subcaption}
\usepackage{verbatim}
\usepackage{indentfirst}
\usepackage{dsfont}
\usepackage{tikz}
\usepackage{enumerate}

\numberwithin{equation}{section}

\def\eqref#1{(\ref{#1})}

\def\1{\sqrt{-1}\:}

\newcommand{\cntrct}                
{\hspace{2pt}\raisebox{1pt}{\text{$\lrcorner$}}\hspace{2pt}}

\renewcommand{\phi}{\varphi}
\renewcommand{\epsilon}{\varepsilon}

\newcounter{Mycounter}[section]
\newcounter{lemma}[section]
\newcounter{claim}[section]

\newcounter{sublemma}[section]

\newcounter{corollary}[section]

\newcounter{theorem}[section]

\newcounter{conjecture}[section]

\newcounter{proposition}[section]

\newcounter{definition}[section]

\newcounter{example}[section]

\newcounter{remark}[section]

\newcounter{problem}[section]

\newcounter{question}[section]

\makeatletter

\@addtoreset{equation}{section}

\@addtoreset{footnote}{section}

\makeatother

\begin{document}
	
\newpage

\title[Darboux-Weinstein theorem for LCS manifolds]{Darboux-Weinstein theorem for locally conformally symplectic manifolds}
\author{Alexandra Otiman and Miron Stanciu}
\address{Institute of Mathematics ``Simion Stoilow'' of the Romanian Academy\\
	21, Calea Grivitei Street, 010702, Bucharest, Romania\\ {\em and}\\
	University of Bucharest, Faculty of Mathematics and Computer Science, 14 Academiei Str., Bucharest, Romania}
\email{alexandra\_otiman@yahoo.com and mirostnc@gmail.com}

\date{\today}

\abstract
A locally conformally symplectic (LCS) form is an almost symplectic form $\omega$ such that a closed one-form $\theta$ exists with $d\omega = \theta \wedge \omega$. We present a version of the well-known result of Darboux and Weinstein in the LCS setting and give an application concerning Lagrangian submanifolds. \\[.1in]

\endabstract

\maketitle

\tableofcontents

\section{Introduction}

The aim of this note is to extend in a natural way the classical Darboux-Weinstein theorem, which we now recall (see e.g. \cite{mcduff}):

\begin{theorem} {\bf (Darboux-Weinstein)} 
\label{thm:DarbouxWeinstein}
Let $M$ be a manifold and $\omega_0, \omega_1$ closed $2$-forms on $M$. Let $Q \subset M$ be a compact submanifold such that $\omega_0$ and $\omega_1$ are nondegenerate and equal on $T_qM$ for all $q \in Q$.

Then there exist $\mathcal{N}_0, \mathcal{N}_1$ neighborhoods of $Q$ and $\varphi : \mathcal{N}_0 \to \mathcal{N}_1$ a diffeomorphism such that 
\[
\varphi^* \omega_1 = \omega_0 \ \ \text{and} \ \ \varphi_{|Q} = id.
\]
\end{theorem}

We are interested in the more general context of locally conformally symplectic (briefly LCS) manifolds, a particular case of almost symplectic manifolds:

\begin{definition}
A manifold $M$ with a non-degenerate two-form $\omega$ is called LCS if there exists a closed one-form $\theta$ such that $d\omega = \theta \wedge \omega$.
\end{definition}

\medskip

The notion first appears as such  in \cite{lib}, it was later studied by J. Lefebvre \cite{lef} and especially I. Vaisman \cite{vaisman}. One can easily see that the  name is justified, as the definition above is equivalent to the existence of an open cover $(U)_\alpha$ and a family of smooth functions $f_\alpha$ on each $U_\alpha$ such that $d (e^{-f_\alpha} \omega  )= 0$ (see \cite{lee}).

\medskip

Our main result reads:

\begin{theorem}
\label{thm:DarbouxWeinsteinLCS}
Let $M$ be a manifold, $\theta_0$ and $\theta_1$ closed $1$-forms and $\omega_0, \omega_1$ $2$-forms on $M$ such that $d_{\theta_i} \omega_i = 0$. Let $Q \subset M$ be a compact submanifold such that $\omega_0$ and $\omega_1$ are nondegenerate and equal on $T_qM$ for all $q \in Q$, and $\theta_{0|TQ} = \theta_{1|TQ}$.

Then there exist $\mathcal{N}_0, \mathcal{N}_1$ neighborhoods of $Q$ and $\varphi : \mathcal{N}_0 \to \mathcal{N}_1$ a diffeomorphism such that 
\[
\varphi^* \omega_1 \sim \omega_0 \ \ \text{and} \ \ \varphi_{|Q} = id.
\]
where by $``\sim"$ we mean conformally equivalent.
\end{theorem}

\medskip

We shall end with an application concerning the behavior of any LCS form near a Lagragian submanifold, thus extending a theorem due to Weinstein \cite{weinstein}.	

\section {Proof of the main theorem}

We shall heavily rely in our own proof on the intricacies of the original Darboux-Weinstein argument, as presented in \cite[Lemma 3.14, pages 93-95]{mcduff}. One of the instruments of both proofs is the so-called Moser Trick, which we therefore explain briefly\footnote{The Moser Trick was  extended to LCS geometry, \cite{bk}, but our proof uses the original  symplectic version.}:

\begin{theorem}
\label{thm:MoserTrick}
Let $M$ be a compact manifold and $(\omega_t)_{0 \le t \le 1}$ a smooth family of symplectic forms on $M$ satisfying 
\[
\frac{d}{dt} \omega_t = d\sigma_t
\]
for $\sigma_t$ varying smoothly. 

Then there is an isotopy $\varphi_t$ such that $\varphi_t^* \omega_t = \omega_0$ with $\varphi_0 = id$.

\begin{proof}[\bf Proof]
Choose the vector fields $Y_t$ uniquely satisfying 
\begin{equation}
\label{eq:MoserFormula}
i_{Y_t} \omega_t = - \sigma_t
\end{equation}
and its integral curves $\varphi_t$ (i.e. $\frac{d}{dt} \varphi_t = Y_t \circ \varphi_t$ and $\varphi_0 = id$), defined overall. What we get is 
\[
\frac{d}{dt}\varphi_t^* \omega_t = \varphi_t^* (\frac{d}{dt} \omega_t + \mathcal{L}_{Y_t} \omega_t) = \varphi_t^* (d\sigma_t + d i_{Y_t} \omega_t) = 0,
\]
hence $\varphi_t^* \omega_t = \omega_0$.
\end{proof}
\end{theorem}

\begin{proof}[\bf Proof of \ref{thm:DarbouxWeinsteinLCS}]
We begin by fixing a Riemannian metric on $M$, which we shall use  to construct a tubular neighborhood of $Q$ in $M$, together with a family of diffeomorphisms representing a deformation retract onto $Q$. 

Take
\[
U_\epsilon = \{ (q, v) \in Q \times TM \ | \ v \in (T_qQ)^\perp \ \text{and} \ \| v \| < \epsilon \},
\]
where the norm is given by the fixed Riemannian metric. Since $Q$ is compact, for a sufficiently small $\epsilon$, the exponential is a diffeomorphism from $U_\epsilon$ to a neighborhood of $Q$ which we denote $\mathcal{M}_0$. We may define 
\[
\phi_t : \mathcal{M}_0 \to \mathcal{M}_0, \ \phi_t (\exp(q, v)) = \exp (q, tv), \ 0 \le t \le 1,
\]
which are diffeomorphisms onto their image, except for $\phi_0$, which collapses the tubular neighborhood onto $Q$. With that in mind, the vector fields
\[
X_t = \frac{d}{dt} \phi_t \circ \phi_t^{-1}
\]
are correctly defined for $0 < t \le 1$ and their integral curves are $\phi_t$.

Let $q \in Q$. We may find $V \supset \overline{U} \supset U \ni q$ small enough that $\omega_0$ and $\omega_1$ are nondegenerate, $\theta_0 = df_0$ and $\theta_1 = df_1$ on $V$; we choose $f_0$ and $f_1$ such that $f_0(q) = f_1(q)$. By our assumptions,
\[
df_{0|TQ} = \theta_{0|TQ} = \theta_{1|TQ} = df_{1|TQ},
\]
so $f_0 = f_1$ on $Q \cap V$. Consider the symplectic forms conformal to $\omega_0$ and $\omega_1$ on $V$:
\begin{align*}
\eta_0 &= e^{-f_0} \omega_0 \\
\eta_1 &= e^{-f_1} \omega_1.
\end{align*}
We can see from the above that $\eta_0$ and $\eta_1$ agree on $T_qM$ for every $q \in Q \cap V$. Let 
\[
W_\delta = \{ (q, v) \in (Q \cap U) \times TM \ | \ v \in (T_qQ)^\perp \ \text{and} \ \| v \| < \delta \};
\]
for $\delta$ sufficiently small (and smaller than the $\epsilon$ determined previously for the entire $Q$), $\exp$ is a diffeomorphism from $W_\delta$ to its image, which we denote $\mathcal{N}$ (note that this is a neighborhood of $Q \cap U$, though not a tubular one, $\mathcal{N} \subset \mathcal{M}_0$ and $\phi_t(\mathcal{N}) \subset \mathcal{N}$). We may also assume, picking a smaller $\delta$ if necessary, that $\mathcal{N} \subset V$ (see the figure below). 

\begin{figure}[H]
	\centering
	\includegraphics[width=0.5\textwidth]{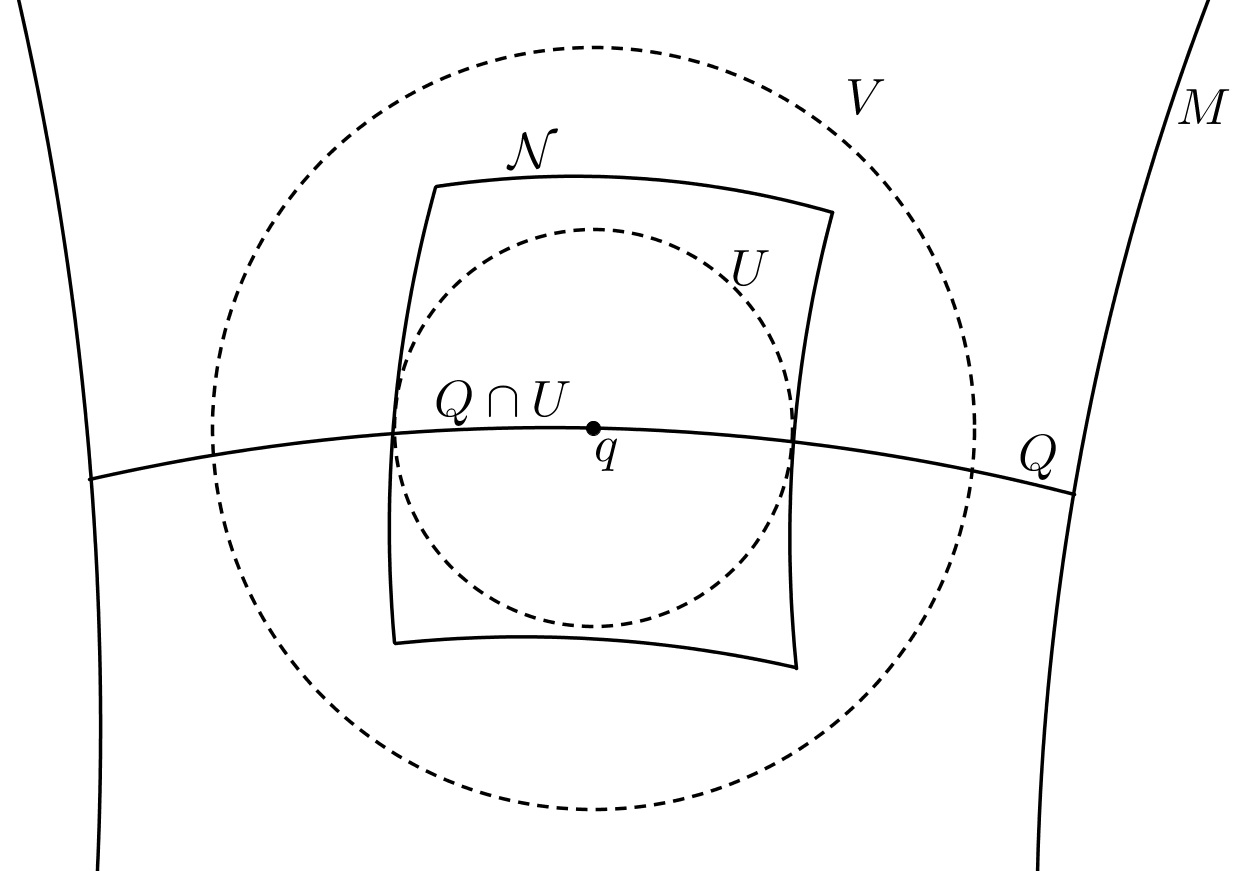}
\end{figure}


Denoting by $\tau := \eta_1 - \eta_0$, we have $\phi_0^* \tau = 0$ and obviously $\phi_1^* \tau = \tau$. Therefore
\[
\tau = \int_0^1 \frac{d}{dt} \phi_t^* (\tau) dt = d \int_0^1 \phi_t^* (i_{X_t} \tau) dt;
\]
let $\rho_t := \phi_t^* (i_{X_t} \tau)$. Explicitly, 
\[
(\rho_t)_p (v) = \tau_{\phi_t(p)} (X_t (\phi_t(p)), \phi_{t*} (v) ) = \tau_{\phi_t(p)} (\frac{d}{dt} \phi_t(p), \phi_{t*} (v) ),
\]
which is correctly defined in $t = 0$. Observe that for $p = q \in Q$, since $\phi_t(q) = q$, we have $(\rho_t)_q = 0$. Taking
\[
\sigma = \int_{0}^{1} \rho_t \ dt,
\]
we have obtained a one-form $\sigma$ on $\mathcal{N}$, null on $Q \cap U$, such that $\eta_1 - \eta_0 = d\sigma$. 

We now turn to the Moser Trick (\ref{thm:MoserTrick}) for the segment of forms $\eta_t = \eta_0 + t (\eta_1 - \eta_0)$, noticing that $\frac{d}{dt} \eta_t = d\sigma$. We may shrink the neighborhood and assume that $\omega_t$ are non-degenerate and that the integral curves obtained are defined on $[0,1]$. We thus get $\varphi : U_q \to U_q'$ (neighborhoods of $Q \cap U$) with $\varphi^* \eta_1 = \eta_0$ and $\varphi_{|Q \cap U} = id$.

We conclude that 
\[
\omega_0 = e^{f_0} \eta_0 = e^{f_0} \varphi^* \eta_1 = e^{f_0 - f_1 \circ \varphi} \varphi^* \omega_1
\]
on the neighborhood $U_q$ of $Q \cap U$.

We have obtained the result we wanted locally on $Q$, by applying (essentially) the Darboux-Weinstein technique on patches of $Q$. Of course, we want the local diffeomorphisms that we have constructed, as well as the conformal factors, to agree on the intersections. This does not usually happen; however, in our case, having the benefit of having used a global instrument (namely, the metric on $M$), we will only need a brief overview of the facts to reach this conclusion. 

We can construct a cover $U_\alpha$ of $Q$ in $M$ such that:
\begin{enumerate}
	\item $\theta_0 = df_0^\alpha$ and $\theta_1 = df_1^\alpha$ on $U_\alpha$;
	\item $f_0^\alpha = f_1^\alpha$ on $Q \cap U_\alpha$;
	\item We have the symplectic forms $\eta_0^\alpha = e^{-f_0^\alpha} \omega_0$ and $\eta_1^\alpha = e^{-f_1^\alpha} \omega_1$ on $U_\alpha$;
	\item There is a $1$-form $\sigma^\alpha$ on $U_\alpha$ with $d\sigma^\alpha = \eta_1^\alpha - \eta_0^\alpha$. More precisely, 
	\[
		\sigma^\alpha = \int_{0}^{1} \phi_t^* i_{X_t} (\eta_1^\alpha - \eta_0^\alpha) dt;
	\] 
	\item The vector field $Y^\alpha_t$ on $U_\alpha$ is uniquely determined by the Moser Formula (\ref{eq:MoserFormula}): 
	\begin{equation}
	\label{eq:Moser2}
	i_{Y^\alpha_t} \eta^\alpha_t = - \sigma^\alpha;
	\end{equation}
	where $\eta^\alpha_t = \eta^\alpha_0 + t (\eta^\alpha_1 - \eta^\alpha_0)$.
	\item Lastly, we have a diffeomorphism $\varphi^\alpha : U_\alpha \to U_\alpha'$ such that 
	\begin{align*}
		\omega_0 &= e^{f_0^\alpha - f_1^\alpha \circ \varphi^\alpha} (\varphi^\alpha)^* \omega_1 \ \text{on} \ U_\alpha \\
		\varphi^\alpha_{|U\alpha \cap Q} &= id,
	\end{align*}
	$\varphi^\alpha$ being the integral curve (at time $t = 1$) of $Y^\alpha_t$.
\end{enumerate}
Note that $\phi_t$ and $X_t$, being a byproduct of the chosen metric, \textbf{are independent of} $\mathbf{\alpha}$, varying only in domain in the above expressions.

On $U_\alpha \cap U_\beta$, we have the following: firstly, since $df_0^\alpha = \theta_0 = df_0^\beta$, 
\[
f_0^\alpha = c_{\alpha \beta} + f_0^\beta.
\]
The same is true of the $f_1$-s:
\[
f_1^\alpha = c'_{\alpha \beta} + f_1^\beta.
\]
However, since $f_0^\alpha = f_1^\alpha$ on $Q \cap U_\alpha$ and $f_0^\beta = f_1^\beta$ on $Q \cap U_\beta$, we conclude that $c_{\alpha \beta} = c'_{\alpha \beta}$. We then immediately get
\begin{align*}
\eta_0^\alpha &= e^{-c_{\alpha \beta}} \eta_0^\beta \\
\eta_1^\alpha &= e^{-c_{\alpha \beta}} \eta_1^\beta,
\end{align*}
so 
\[
\eta^\alpha_t = e^{-c_{\alpha \beta}} \eta^\beta_t \ \text{and} \ \sigma^\alpha = e^{-c_{\alpha \beta}} \sigma^\beta.
\]

We now see clearly from (\ref{eq:Moser2}) that, on $U_\alpha \cap U_\beta$, the vector fields $Y^\alpha_t$ and $Y^\beta_t$ satisfy the same formula, and must be equal. Then $\varphi^\alpha = \varphi^\beta$ on $U_\alpha \cap U_\beta$, and we can glue them to a global diffeomorphism
\[
\varphi : \mathcal{N}_0 := \bigcup\limits_{\alpha} U_\alpha \to \mathcal{N}_1 := \bigcup\limits_{\alpha} U'_\alpha 
\]
with $\varphi_{|Q} = id$ and 
\[
\omega_0 = e^{f_0^\alpha - f_1^\alpha \circ \varphi} \varphi^* \omega_1 \ \text{on} \ U_\alpha, \ \forall \alpha.
\]
However, it is clear now that the conformal factors are also equal on the intersections:
\[
f_0^\alpha - f_1^\alpha \circ \varphi = c_{\alpha \beta} + f_0^\beta - (c_{\alpha \beta} + f_1^\beta) \circ \varphi = f_0^\beta - f_1^\beta \circ \varphi,
\]
and we have reached our conclusion.
\end{proof}

\medskip

\begin{remark}
\label{rem}
The condition of equality on $T_qM$ of the two LCS forms might seem a bit restrictive. Nevertheless, there are a few cases where it may be lessened to equality on $TQ$, for instance if $Q$ is a point (where the conclusion is an easy consequence of the classical Darboux theorem) or if $Q$ is Lagragian for both $\omega_0$ and $\omega_1$. In the latter case, the proof in \cite[Theorem 8.4, pages 48-49]{daSilva} can be readily adapted to the LCS case, thus reducing the problem to \ref{thm:DarbouxWeinsteinLCS}.
\end{remark}

\section{An application}

In the symplectic case, the last remark has as consequence the following theorem, describing any symplectic form around a Lagrangian submanifold in terms of the standard symplectic form on its cotangent bundle. We state the precise result below, due to Weinstein \cite[Theorem 6.1, pages 338-339]{weinstein}:

\begin{theorem}
\label{thm:WeinsteinTubular}
Let $(M, \omega)$ be a symplectic manifold and $Q \subset M$ a compact Lagrangian submanifold. 

Then there exists a neighborhood $\mathcal{M}$ of $Q$, a neighborhood $\mathcal{N}$ of the zero section in $T^*Q$ and a diffeomorphism $\varphi: \mathcal{M} \to \mathcal{N}$ such that 
\[
\varphi^* \omega_0 = \omega,
\]
where $\omega_0$ is the standard symplectic form on $T^*Q$.
\end{theorem}

\medskip

There is an analogue of this result in the LCS case, which uses the LCS structures of the cotangent bundle introduced by S. Haller and T. Rybicki in \cite{haller}: take $\theta$ a closed one-form on $Q$ and $\eta$ the Liouville form on $T^*Q$. Then it can be proven that 
\[
\omega_\theta = d \eta - \pi^* \theta \wedge \eta
\]
is LCS with the Lee form $\pi^* \theta$. It can also be easily seen that the zero section is then Lagrangian. Note that $\omega$ is globally conformally symplectic if and only if $\theta$ is exact.

\medskip

We can now state our extension of the previous theorem to LCS manifolds:

\begin{theorem}
\label{thm:WeinsteinTubularLCS}
Let $(M, \omega)$ be an LCS manifold with Lee form $\theta$ and $Q \subset M$ a compact Lagrangian submanifold. 

Then there exists a neighborhood $\mathcal{M}$ of $Q$, a neighborhood $\mathcal{N}$ of the zero section in $T^*Q$ and a diffeomorphism $\varphi: \mathcal{M} \to \mathcal{N}$ such that 
\[
\varphi^* \omega_\theta = \omega,
\]
where $\omega_\theta$ is the LCS form described above.
\end{theorem}

\begin{proof}[\bf Proof]
We first wish to transport the form $\omega_\theta$ from $T^*Q$ to a neighborhood of $Q$ in $M$. Fix a Riemannian metric on $M$; we then have a canonical isomorphism of vector bundles between $(TQ)^\perp$ and $T^*Q$, given by:
\[
(TQ)^\perp \ni (q, v) \mapsto (q, w^*)
\]
\noindent where $w$ is uniquely found by $g(v, \cdot) = \omega (w, \cdot)$ and $w^* = g(w, \cdot)$ (the key point in this identification is the fact that $T_qQ$ is Lagrangian in $T_qM$ for each $q \in Q$).

Furthermore, by means of the exponential map, a neighborhood of $Q$ in $M$ is diffeomorphic to a neighborhood of the zero section in $(TQ)^\perp$. Consequently, we can transport the form $\omega_\theta$ to a neighborhood $\mathcal{U}$ of $Q$. 

We need only remember that $Q$ is also Lagrangian for this new form and apply theorem \ref{thm:DarbouxWeinsteinLCS} in light of  \ref{rem} to complete the proof.
\end{proof}

\end{document}